\documentclass[a4paper,11pt]{article}
\usepackage[T1]{fontenc}
\usepackage[utf8]{inputenc}
\usepackage{lmodern}
\usepackage[margin=2.5cm]{geometry}
\usepackage{microtype}
\usepackage{amsmath,amsthm,amssymb,mathtools}
\usepackage{enumitem,bm}
\usepackage{braket, booktabs,longtable,caption,subcaption,tabu,multirow, makecell,pdflscape}
\usepackage{xparse}
\usepackage{multirow}
\usepackage{color,colortbl}
\usepackage{tikz}
\usepackage{comment}
\usetikzlibrary{cd,calc}
\usepackage{hyperref}
\usepackage[hyphenbreaks]{breakurl}
\newcolumntype{B}[1]{>{\begin{minipage}[b]{#1}\raggedright{}}c<{\end{minipage}\minrowheight}}
\newcommand{\minrowheight}{\rule{0pt}{8ex}\relax}

\usepackage{algorithm}
\usepackage[noend]{algpseudocode}

\makeatletter
\def\algbackskip{\hskip-\ALG@thistlm}
\makeatother

\theoremstyle{plain}
\newtheorem{Thm}{Theorem}

\newtheorem{Lem}[Thm]{Lemma}
\newtheorem{Cor}[Thm]{Corollary}

\theoremstyle{definition}
\newtheorem{Ex}[Thm]{Example}
\newtheorem{Def}[Thm]{Definition}

\theoremstyle{remark}

\makeatletter
\def\@tocline#1#2#3#4#5#6#7{\relax
  \ifnum #1>\c@tocdepth 
  \else
    \par \addpenalty\@secpenalty\addvspace{#2}%
    \begingroup \hyphenpenalty\@M
    \@ifempty{#4}{%
      \@tempdima\csname r@tocindent\number#1\endcsname\relax
    }{%
      \@tempdima#4\relax
    }%
    \parindent\z@ \leftskip#3\relax \advance\leftskip\@tempdima\relax
    \rightskip\@pnumwidth plus4em \parfillskip-\@pnumwidth
    #5\leavevmode\hskip-\@tempdima
      \ifcase #1
       \or\or \hskip 1em \or \hskip 2em \else \hskip 3em \fi%
      #6\nobreak\relax
    \hfill\hbox to\@pnumwidth{\@tocpagenum{#7}}\par
    \nobreak
    \endgroup
  \fi}
\makeatother
\numberwithin{Thm}{section}
\allowdisplaybreaks

\newcommand{\Hom}{\operatorname{Hom}}
\newcommand{\Bl}{\operatorname{Bl}}
\newcommand{\Nef}{\operatorname{Nef}}
\newcommand{\Eff}{\operatorname{Eff}}
\newcommand{\Mov}{\operatorname{Mov}}
\newcommand{\Proj}{\operatorname{Proj}}
\newcommand{\Cl}{\operatorname{Cl}}

\newcommand{\Cox}{\operatorname{Cox}}
\newcommand{\Spec}{\operatorname{Spec}}

\newcommand{\GL}{\operatorname{GL}}

\newcommand{\Int}{\operatorname{Int}}

\newcommand{\wt}{\operatorname{wt}}

\newcommand\rat{\dashrightarrow}

\newcommand\actL[1]{\multirow{2}{*}{$#1:$} \close}
\newcommand\actR[1]{\close \multirow{2}{*}{\bigg)#1}}
\newcommand\lBr{\rule{0pt}{2.5ex} \multirow{2}{*}{\bigg(} \close}

\newcommand\close{\!\!\!\!}

\usepackage[all]{xy}

\title{Sarkisov links from toric weighted blowups of $\mathbb{P}^3$ and $\mathbb{P}^4$ at a point}
\author{Tiago Duarte Guerreiro}
\date{}

\begin{document}

\maketitle

\begin{abstract}
    We study Sarkisov links initiated by the toric weighted blowup of a point in $\mathbb{P}^3$ or $\mathbb{P}^4$ using variation of GIT. We completely classify which of these initiate Sarkisov links and describe the links explicitly. Moreover, if $X$ is the weighted blowup of $\mathbb{P}^d$ at a point, we give a simple criterion in terms of the weights of the blowup that characterises when $X$ is weak Fano.
    
\end{abstract}


\section{Introduction}

The aim of this paper is to explicitly describe all possible toric Sarkisov links initiated by a toric weighted blowup of a point in $\mathbb{P}^3$ or $\mathbb{P}^4$. We do this using variation of GIT (vGIT). According to the Sarkisov program, see \cite{CortiSP,SP}, a birational map $X\dashrightarrow Y/S $ from a Fano variety of Picard rank 1 to a Mori fibre space can be decomposed into a finite sequence of elementary Sarkisov links starting with the blowup of a centre in $X$. Therefore it is natural to try to understand Sarkisov links explicitly and a lot of work has been done in that direction. See \cite{CPR, cortimella, 2raybrown, okadaI, righyp, CampoSarkisov, okadacod4,mythesis, erik, HamidVanyaJihun}.

In \cite{LamyBlanc}, Blanc and Lamy classify the curves in $\mathbb{P}^3$ whose blowup $\varphi \colon X \rightarrow \mathbb{P}^3$ produces a weak Fano 3-fold. Then $X$ is a smooth manifold of Picard rank 2 and hence equipped with two extremal contractions: One for which $-K_X$ is relatively ample and another one associated with the linear system $|-mK_X|$, where $m\gg 0$. The latter is a small contraction if the movable cone of $X$ is strictly larger than the Nef cone of $X$. In that case there is a Sarkisov link, which the authors describe.

In this paper we take a similar approach. Let $Y=\mathbb{P}^3$ or  $\mathbb{P}^4$. Given a toric weighted blowup of a point  $\varphi \colon T \rightarrow Y$ we decide whether $T$ is in the Mori category of terminal $\mathbb{Q}$-factorial varieties. If it is, we compute the decomposition of the movable cone of $T$ into Nef chambers. Notice that $T$ does not need to be a weak Fano for a Sarkisov link to exist. In fact, we only need that the class of $-K_T$ is in the interior of the movable cone of $T$.

Our main results are the following:
\begin{Thm}
Let $\varphi \colon T \rightarrow \mathbb{P}^3$ be the toric $(1,a,b)$-weighted blowup of a point. Then $\varphi$ initiates a toric Sarkisov link from $\mathbb{P}^3$ if and only if, up to permutation, 
\[
(a,b) \in \{(1,1), (1,2), (2,3), (2,5) \}.
\]
\end{Thm}

\begin{Thm} \label{thm:2}
Let $\varphi \colon T \rightarrow \mathbb{P}^4$ be the toric $(a,b,c,d)$-weighted blowup of a point. Then $\varphi$ initiates a toric Sarkisov link from $\mathbb{P}^4$ if and only if $(a,b,c,d)$ is one of 421 quadruples up to permutation.
\end{Thm}

 Throughout this paper we only consider toric  $(w_1,\ldots,w_d)$-weighted blowups. In particular, we do not describe all blowups which are locally analytically $(1,a,b)$-weighted blowups as in Theorem \ref{thm:kawakita}.

All the links are explicitly described. Given the number of cases in Theorem \ref{thm:2}, we use the help of a Maple program written by the author. See Section \ref{sec:code}.

\paragraph{Acknowledgements} The author would like to thank Hamid Abban, Ivan Cheltsov and Erik Paemurru for many conversations and interest. This work has been supported by the EPSRC grant ref.\ EP/V048619/1.

\section{Preliminaries}

\paragraph{Weighted blowups and terminal singularities.}

We recall the definition of weighted blowups:
\begin{Def}
Let $\alpha =(\alpha_1,\ldots,\alpha_n)$ be positive integers and define the $\mathbb{C}^*$-action on $\mathbb{C}^{n+1}$ by $\lambda \cdot (u,x_1,\ldots,x_n) = (\lambda^{-1}u,\lambda^{\alpha_1}x_1,\ldots,\lambda^{\alpha_n}x_n)$. Let $T=(\mathbb{C}^{n+1}\setminus \mathcal{Z}(x_1,\ldots,x_n))/\mathbb{C}^*$. Then, the morphism  $\varphi \colon T \rightarrow \mathbb{C}^n$ given by $(u,x_1,\ldots,x_n) \mapsto (u^{\alpha_1}x_1,\ldots,u^{\alpha_n}x_n)$ is called the \textbf{weighted blowup} of $\mathbb{C}^n$ at the origin. 
\end{Def}

The following theorem says that in dimension three, a divisorial extraction centred at a smooth point to a terminal $\mathbb{Q}$-factorial variety is a weighted blowup.

\begin{Thm}[{\cite[Thm.~2.2]{kawakita}}] \label{thm:kawakita}
Let $Y$ be a $\mathbb{Q}$-factorial normal variety of dimension three with only terminal singularities, and let $f \colon (Y \supset E ) \rightarrow ( X \ni P)$ be an algebraic germ of a divisorial contraction which contracts its exceptional divisor $E$ to a smooth point $P$. Then $f$ is a weighted blow-up. More precisely, we can take local coordinates $x, y, z$ at $P$ and coprime positive integers $a$ and $b$, such that $f$ is the weighted blow-up of $X$ with its weights $\wt(x, y, z) = (1, a, b)$ We call $f$ a \textbf{Kawakita blowup} of $p$.
\end{Thm}



We recall the definition of a cylic quotient singularity. Let $\bm{\mu}_r$ be the cyclic group of $r$th roots of unity. Define the action of $\bm{\mu}_r$ on $\mathbb{C}^n$ by
\begin{align*}
\bm{\mu}_r\times \mathbb{C}^n &\longrightarrow \mathbb{C}^n \\
(\epsilon, (x_1,\ldots,x_n)) &\longmapsto (\epsilon^{a_1}x_1,\ldots,\epsilon^{a_n}x_n)
\end{align*} 
where $\epsilon$ is a primitive $r$ root of unity and $a_i$ are integers. The cyclic quotient singularity associated to the action above is the quotient $X=\mathbb{C}^n/\bm{\mu}_r$ which we denote by
\[
\frac{1}{r}(a_1,\ldots,a_n).
\] 
In dimension three, terminal extractions centred at terminal cyclic quotient singularities are also classified:
\begin{Thm}[\cite{kawamata}] \label{thm:kawamata}
Let $p\sim \frac{1}{r}(1,a,r-a)$ be a germ of a 3-fold terminal quotient singularity and $f \colon E \subset Y \rightarrow  \Gamma \in X$ a divisorial contraction centred at $\Gamma \ni p$, then $\Gamma = p$ and $f$ is the $(1,a,r-a)$-weighted blowup of $p$. We call $f$ the \textbf{Kawamata blowup} of $p$.
\end{Thm}

\begin{Ex}
The statement of theorem \ref{thm:kawamata} fails in higher dimension. Let $X=\mathbb{P}(1,5,7,8,10)$ with homogeneous variables $x,\,y,\,z,\,t,\,v$. Let $\Gamma \colon (x=z=t=0) \sim \frac{1}{5}(1,2,3) \in X$. Then, the $(1,2,3)$-weighted blowup of $\Gamma$ is terminal. Moreover, the point $p = \Gamma_{|z=0}$ is a cyclic quotient singularity of type $\frac{1}{5}(1,2,3,0)$. Let $m \in \mathbb{Z}_{>0}$ and consider the family $\varphi_m$ of $(1,2,3,5m)$-weighted blowup of $p$. Then $\varphi_m$ is terminal.   
\end{Ex}

Fix $d \in \mathbb{N}$. The classification of terminal $d$-fold singularities is a wide open problem settled completely only for $d\leq 3$. If we restrict to cyclic quotient singularities the problem has been solved only very recently for $d=4$ in \cite{4foldsing}. Likewise, the classification of terminal extractions from these is also completely open, even for $d=3$. A fundamental application of this knowledge is the possibility, in principle, of constructing Sarkisov links from higher dimensional Fano varieties.  Although terminal singularities are not classified in general, the following is a criterion to check when a \textit{cyclic quotient singularity} of any dimension is terminal.

\begin{Thm} \label{thm:YPG} \cite[Theorem~4.11]{YPG}\label{thm:terminalsing}
A cyclic quotient singularity $\frac{1}{r}(a_1,\ldots,a_n)$ is terminal if and only if 
\[
\sum_{i=1}^n\overline{l a_i}>r, \quad \text{for}\,\, l = 1 \ldots, r-1
\]
where $\overline{\phantom{a_i}}$ denotes smallest residue $\bmod r$.
\end{Thm}

Using theorem \ref{thm:terminalsing}, one can check effectively whether a weighted blowup is terminal. The following is a consequence of \cite[Corollary~2.6]{Sankaran}.

\begin{Thm} \label{thm:wtbup}
The $(\alpha_1,\ldots,\alpha_n)$-weighted blowup of $\mathbb{C}^n$ at the origin has terminal singularities if and only if 
\[
\frac{1}{V}(\alpha_1,\ldots,\alpha_n)
\]
is terminal where $V=-1+\sum_{i=1}^n \alpha_i$.
\end{Thm}

\paragraph{The weighted blowup of  $\mathbb{P}^d$ at a point.}

Let $\varphi \colon T \rightarrow \mathbb{P}^d$ be the toric $(\alpha_1,\ldots,\alpha_d)$-weighted blowup of $\mathbb{P}^d$ at the coordinate point $\mathbf{p}$ which we fix to be $(1\colon 0 \colon \ldots \colon 0)$ without loss of generality. Since we are blowing up a toric variety along a torus invariant ideal it follows that $T$ toric of rank 2 and $\Cl(T)= \mathbb{Z}[H]+\mathbb{Z}[E]$, where $H=\varphi^*\mathcal{O}_{\mathbb{P}^d}(1)$  and $E = \Phi^{-1}(\mathbf{p})\simeq \mathbb{P}(\alpha_1,\ldots,\alpha_d)$ is the exceptional divisor.\ Notice that $\mathbf{p}$ is not in the support of $\mathcal{O}_{\mathbb{P}^d}(1)$.\ $T$ is $\mathbb{Q}$-factorial, since any Weil prime divisor on $T$ is a linear combination of $\mathbb{Q}$-Cartier divisors. By \cite[Theorem~2.2]{toricproj}, $T$ is projective since it is complete. The Cox ring of $T$ is
\begin{equation*}
\Cox(T) = \bigoplus_{m,n \in \mathbb{Z}_{\geq 0}} H^0(T,mH-nE) \; .
\end{equation*}
Since $T$ is toric, the Cox ring of $T$ is isomorphic to the (bi)-graded polynomial ring $\mathbb{C}[u,x_0,\ldots,x_d]$.  If 
\[
(u,x_0, \ldots, x_d) \mapsto (x_0 \colon u^{\alpha_1}x_1 \colon \ldots \colon u^{\alpha_d}x_d)
\]
is the weighted blowup $\varphi \colon T \rightarrow \mathbb{P}^d$ at $\mathbf{p}$, then we say that $x_i$ has bidegree $(1, \alpha_i)$. The toric variety $T$ is then represented by the matrix
\begin{equation*}
\begin{array}{cccc|ccccc}
             &       & u     & x_0 & x_1 & \ldots & x_d & \\
\actL{T}   &  \lBr &  0 & 1 & 1 & \ldots & 1 &   \actR{}\\
             &       & -1 & 0 & \alpha_1 & \ldots & \alpha_d &  
\end{array}
\end{equation*}
where the vertical bar represents the irrelevant ideal which in this case is $(u,x_0) \cap (x_1,\ldots,x_d)$ and each row is a representation of a $\mathbb{C}^*$-action on $\mathbb{C}^{d+2}$. In other words, the variety $T$ is defined by the geometric quotient 
\[
T=\frac{\mathbb{C}^{d+2}\setminus \mathcal{Z}((u,x_0)\cap(x_1,\ldots,x_d))}{\mathbb{C}^*\times \mathbb{C}^*}.
\]
See \cite[Theorem~2.1]{Coxhom}. The variety $T$ is built up from $2 \times d$ affine patches. These are of the form
\[
(ux_i\not = 0) \quad \text{and} \quad (x_0x_i\not =0)
\]
where $1\leq i\leq d$. We have,
\begin{align*}
    (ux_i\not =0) &\simeq \Spec \mathbb{C}\Big[u,x_0,\ldots,x_d,\frac{1}{u},\frac{1}{x_i}\Big]^{\mathbb{C}^* \times\mathbb{C}^*} \\
    &\simeq \Spec \mathbb{C}\Big[\frac{x_0}{x_iu^{\alpha_i}},\frac{x_1}{x_i}u^{\alpha_1-\alpha_i},\ldots,\frac{x_d}{x_i}u^{\alpha_d-\alpha_i}\Big]\\
      &\simeq \mathbb{A}^d
\end{align*}
On the other hand,
\begin{align*}
    (x_0x_i\not =0) &\simeq \Spec \mathbb{C}\Big[u,x_0,\ldots,x_d,\frac{1}{x_0},\frac{1}{x_i}\Big]^{\mathbb{C}^* \times\mathbb{C}^*} \\
      &\simeq \frac{1}{\alpha_i}(-1,\alpha_1,\ldots, \widehat{\alpha_i},\ldots,\alpha_d)
\end{align*}

\begin{Ex}
Consider the 3-fold given by
\begin{equation*}
\begin{array}{cccc|ccccc}
             &       & u     & x_0 & x_1 & x_2 & x_3 & \\
\actL{T}   &  \lBr &  0 & 1 & 1 & 1 & 1 &   \actR{.}\\
             &       & -1 & 0 & 1 & 1 & 3 &  
\end{array}
\end{equation*}
Then, the affine patch $(x_0x_3\not = 0)$ is isomorphic to 
\begin{align*}
    \Spec \mathbb{C}\Big[u,x_0,\ldots,x_d,\frac{1}{x_0},\frac{1}{x_i}\Big]^{\mathbb{C}^* \times\mathbb{C}^*} & = \Spec \mathbb{C}\Big[\frac{ux_1}{x_0},\frac{ux_2}{x_0},\frac{u^3x_3}{x_0},\frac{x_1^3}{x_3x_0^2},\frac{x_1^2x_2}{x_3x_0^2},\frac{x_1x_2^2}{x_3x_0^2},\frac{x_2^3}{x_3x_0^2}\Big]\\
    &=\Spec\mathbb{C}[xz,yz,z^3,x^3,x^2y,xy^2,y^3]\\
    &=\Spec\mathbb{C}[x,y,z]^{\bm{\mu_3}}
\end{align*}
where the action of  $\bm{\mu_3}$ on $\mathbb{C}^3$ is given by $(x,y,z) \mapsto (\mu x, \mu y, \mu^2 z)$. That is, $(x_0x_3\not = 0) \simeq \frac{1}{3}(1,1,2)$.
\end{Ex}

Without loss of generality we can assume $\alpha_i \leq \alpha_{i+1}$. Let $N^1(T)$ be the finite dimensional vector space of $\mathbb{Q}$-divisors modulo numerical equivalence.  The effective cone of $T$, $\Eff(T) \subset N^1(T)$, is generated by the rays $\mathbb{R}_+[E]+\mathbb{R}_+[H-\alpha_d E]$ and it has a decomposition into smaller chambers. By \cite[Proposition~1.11]{mdsGIT}, there is a finite number of birational maps $f_i \colon T \dashrightarrow Y_i$, with $Y_i$  Mori dream spaces, such that 
\[
\Eff(T)=\bigcup_{i}\mathcal{C}_i,
\]
where
\begin{align*}
\mathcal{C}_i&=\bigcup_{i}f_i^*(\Nef(Y_i))+\mathbb{R}_+[E] && \text{if}\,\, \alpha_{d-1}=\alpha_d\\
\mathcal{C}_i&=\bigcup_{i}f_i^*(\Nef(Y_i))+\mathbb{R}_+[E]+\mathbb{R}_+[H-\alpha_dE] && \text{if}\,\, \alpha_{d-1}<\alpha_d.
\end{align*}

Recall that the cone of Movable divisors of $T$, $\Mov(T) \subset \Eff(T)$, is the cone of divisors $D$ for which the base locus of $|D|$ has codimension at least 2 in $T$. There is a finite number of small birational maps $f_j \colon T \dashrightarrow Y_j$ for which $\Mov(T)=\bigcup_{j}f_j^*(\Nef(Y_j))$. We have
\[
\Nef(T)=\mathbb{R}_+[H]+\mathbb{R}_+[H-\alpha_1 E] \subset \mathbb{R}_+[H]+\mathbb{R}_+[H-\alpha_{d-1}E]=\Mov(T).
\]

\begin{Thm} \label{thm:weak}
Let $\varphi \colon T \rightarrow \mathbb{P}^d$ be the toric $(\alpha_1,\ldots,\alpha_d)$-weighted blowup of a point where $d\geq 2$. Suppose that $T$ is a normal projective variety. Then $-K_T$ is big. Moreover,  $T$ is weak Fano if and only if 
\[
\frac{\sum_{i=1}^d\alpha_i-1}{d+1}\leq \min_{1 \leq i\leq d}{\alpha_i}
\]
with equality if and only if $-K_T$ is nef but not ample.
\end{Thm}

\begin{proof}
The variety $T$ is toric since it is the blowup of a toric variety along a torus invariant ideal.
Since $-K_T$ is the sum of the torus invariant divisors, \cite[Theorem~8.2.3]{Cox}, it follows that
\begin{equation*}
-K_T \sim (d+1)H-\Big(\sum_{i=1}^d\alpha_i -1\Big)E
\end{equation*}
and that $-K_T$ is big, since it is a sum of effective divisors. Let $\alpha= \min_{1 \leq i\leq d} {\alpha_i}$. Then the Nef cone of $T$ is $\mathbb{R}_+[H]+\mathbb{R}_+[H-\alpha E]$. We can write $-\alpha K_T$ in terms of the generators of $\Nef(T)$ as 
\[
-\alpha K_T \sim \Big((d+1)\alpha+1-\sum_{i=1}^d\alpha_i\Big)H+\Big(\sum_{i=1}^d\alpha_i -1\Big)(H-\alpha E).
\]
Since $\sum_{i=1}^d\alpha_i -1 >0$ it follows that $-K_T$ is nef if and only if 
\[
(d+1)\alpha+1-\sum_{i=1}^d\alpha_i \geq 0
\]
as we wanted to show.
\end{proof}

\begin{Cor}
Let $\varphi \colon T \rightarrow \mathbb{P}^3$ be the $(1,a,b)$-Kawakita blowup of a point. Then  
\begin{itemize}
    \item $T$ is Fano if and only if $(a,b) \in \{(1,1), (1,2) \}$;
    \item $T$ is a weak Fano but not Fano if and only if $(a,b) \in \{(1,3) \}$.
\end{itemize}
\end{Cor}

\begin{proof}
By theorem \ref{thm:weak}, we only have to check the inequality $4-a-b\geq 0$. Since $\gcd(a,b)=1$ we have $(a,b) \in \{(1,1),(1,2),(1,3) \}$ with equality only in the case $(a,b)=(1,3)$.
\end{proof}

\begin{Thm} \label{thm:ineq}
Let $\varphi \colon T \rightarrow \mathbb{P}^d$ be the toric    $(\alpha_1,\ldots,\alpha_d)$-weighted blowup of a point where $d\geq 2$. Suppose that $T$ is a normal weak Fano variety. Then,
\[
\frac{\sum_{i=1}^d \alpha_i-1}{d+1} <\sqrt[d]{\Pi_{i=1}^d \alpha_i}  \leq \frac{\sum_{i=1}^d \alpha_i}{d} 
\]
Moreover, the second inequality is strict unless $\varphi$ is the ordinary blowup of a point.
\end{Thm}

\begin{proof}
Since $T$ is a weak Fano it follows from \cite[Theorem~2.2.16]{posI} that $(-K_T)^d>0$. We have, 
\begin{align*}
(-K_T)^d &= \bigg((d+1)H-\Big(\sum_{i=1}^d\alpha_i -1\Big)E\bigg)^d \\
 &=(d+1)^d-\frac{\Big(\sum_{i=1}^d\alpha_i -1\Big)^d}{\Pi_{i=1}^d \alpha_i}.
\end{align*}
Notice that since the point we blowup is not in the support of $\mathcal{O}_{\mathbb{P}^d}(1)$, we have $E^k \cdot H^{d-k}=0$ for any $0<k<d$. Hence,
\[
\frac{\sum_{i=1}^d \alpha_i-1}{d+1} <\sqrt[d]{\Pi_{i=1}^d \alpha_i}. 
\]
The other inequality is the AM-GM inequality and it is known that equality holds if and only if $\alpha:=\alpha_1=\cdots =\alpha_d$. In that case suppose that $\alpha>0$. Then, $T$ is singular along the exceptional divisor which is impossible since $T$ is normal.
\end{proof}

\section{Sarkisov Links}

Since $T$ is a Mori dream space, a divisorial extraction $\varphi \colon T \rightarrow \mathbb{P}^d$ initiates a Sarkisov link if and only if the class of the anticanonical divisor of $T$ is in the interior of the movable cone of $T$ and any small $\mathbb{Q}$-factorial modification $T\dashrightarrow T'$ is terminal. See for instance \cite[Lemma~2.9]{hamidquartic}.
In that case, the construction of an elementary Sarkisov link for $T$ is naturally divided into two main cases, depending roughly on its Mori chamber decomposition. 
Namely the behaviour of its movable cone of divisors near the boundary of the effective cone of divisors of $T$.\
\begin{enumerate}
    \item \label{item: fibr} \textbf{Fibration}: the class of $H-\alpha_{d-1}E$ is not big. In this case $\alpha_{d-1}=\alpha_d$.
    \item \label{item: div contr} \textbf{Divisorial Contraction}: the class of $H-\alpha_{d-1}E$ is big. In this case $\alpha_{d-1}<\alpha_{d}$ and therefore $H-\alpha_{d-1}E \not \sim_{\mathbb{Q}} H-\alpha_{d}E$.\ Moreover, we contract the effective divisor $H-\alpha_{d}E$ to an $r$-dimensional variety where $r=|\{\alpha_i\,\, |\,\, \alpha_i = \alpha_{d-1}\}|-1$.
\end{enumerate}

 A Sarkisov link initiated by $\varphi$ is of the form

\begin{equation*}
\begin{tikzcd} 
& T \arrow[swap]{dl}{\varphi} \arrow[swap]{dr}{\alpha_0} \arrow[dashed]{rr}{\tau_0} & &  T_1\arrow[swap]{dr}{\alpha_1} \arrow{dl}{\beta_0} \arrow[dashed]{r}{\tau_1} & \cdots \arrow[dashed]{r}{\tau_n}   & T'\arrow{dr}{\varphi'}\arrow{dl}{\beta_{n-1}}  & \\
 \mathbb{P}^d &  & \mathcal{F}_0 &  & \cdots  &  & \mathcal{F}'  
\end{tikzcd}
\end{equation*}
and $\dim \mathcal{F}' \leq \dim T'$, with equality if and only if we are in the second case.\  Moreover, since $\varphi$ is a projective birational toric morphism, the diagram above is a toric Sarkisov link. See for instance \cite[Theorem~0.2]{ReidToricMMP} or \cite[Theorem~4.1]{2raybrown}.

\begin{Ex} We give an example of how a small modification arises from variation of GIT from a $\mathbb{C}^*$ action. As mentioned in \cite{reidflip}, Mori flips arise naturally in this context.

For a Mori Dream Space, the GIT chambers and Mori chambers coincide, see \cite[Theorem~2.3]{mdsGIT}. Hence, the 2-ray game on a rank 2 toric variety $T$ can be obtained by variation of GIT as explained in \cite[Section~4]{2raybrown}.  In this case, let $T$ be the $\mathbb{P}^1$-bundle $T= \Proj_{\mathbb{P}^1} \mathcal{E}$, where 
\[
\mathcal{E}=\mathcal{O}_{\mathbb{P}^1}\oplus\mathcal{O}_{\mathbb{P}^1}(1)\oplus\mathcal{O}_{\mathbb{P}^1}(1).
\]
Then, $T$ is  
\[
\begin{array}{cccc|cccc}
             &       & y_1  & y_2 &   t & x_1 & x_2  & \\
\actL{T}   &  \lBr &  0 & 0 &   1 & 1 & 1  &   \actR{}\\
             &       & 1 & 1 &   0 & -1 & -1 &  
\end{array}
\] 
with Cox ring  $\mathbb{C}[y_1,y_2,t,x_1,x_2]$
and irrelevant ideal  $(y_1,y_2) \cap (t,x_1,x_2)$ as shown by the vertical bar. That is, we have an action of $(\mathbb{C}^*)^2$ on $\mathbb{C}^5$ given by
\[
(\lambda,\mu) \cdot (y_1,y_2,t,x_1,x_2) = (\mu y_1,\mu y_2, \lambda t, \lambda\mu^{-1} x_1,\lambda \mu^{-1}x_2) 
\]
The GIT chambers of $T$ are 

\begin{figure}[h]%
\centering
\begin{tikzpicture}[scale=3,font=\small]
  \coordinate (A) at (0, 0);
  \coordinate [label={left:$(0,1)$}] (E) at (0,0.5);
  \coordinate [label={above:$(1,0)$}] (K) at (0.5,0);
	\coordinate [label={right:$(1,-1)$}] (5) at (0.5,-0.5);
	\draw  (A) -- (E);
  \draw  (A) -- (K);
	\draw (A) -- (5);
	\end{tikzpicture}
\caption{A representation of the GIT chamber decomposition of $T$.}%
\label{fig:git}%
\end{figure}
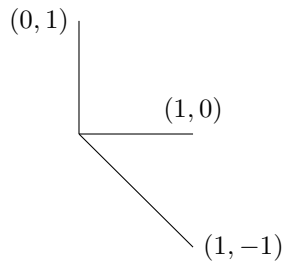
Choosing the character $(0,1)$ constructs a variety isomorphic to $\mathbb{P}^1$. Notice we are implicitly using the fact that $\mathbb{Z}\oplus \mathbb{Z} $ and $\Hom_{\mathbb{Z}}((\mathbb{C}^*)^2,\mathbb{C}^*)$ are isomorphic abelian groups.

This is given by 
\[
\Proj \bigoplus_{m\geq 0} H^0(T,\mathcal{O}(0,m)) = \Proj \mathbb{C}[y_1,y_2] \simeq \mathbb{P}^1
\]
and similarly with the ray generated by $(1,-1)$. The ray $(1,0)$ constructs the non-$\mathbb{Q}$-factorial variety
\[
\Proj \bigoplus_{m\geq 0} H^0(T,\mathcal{O}(m,0)) = \Proj \mathbb{C}[t,y_1x_1,y_1x_2,y_2x_1,y_2x_2] \simeq (u_{11}u_{22}-u_{12}u_{21}=0) \subset \mathbb{P}^4
\]
where $u_{ij} := y_ix_j$. This is a quadric cone $Q$ whose only singular point is $p=(1:0:0:0:0)$. On the other hand, choosing characters in the interior of the two GIT chambers constructs varieties related by a small $\mathbb{Q}$-factorial modification. Choosing a character in the interior of the cone generated by $\left(\begin{smallmatrix} 0 \\ 1 \end{smallmatrix} \right)$ and $\left(\begin{smallmatrix} 1 \\ 0 \end{smallmatrix} \right)$ constructs a variety isomorphic to $T$ since these generate the Mori cone of $T$. On the other hand, choosing a character in the interior of the cone generated by $\left(\begin{smallmatrix} 1 \\ 0 \end{smallmatrix} \right)$ and $\left(\begin{smallmatrix} 1 \\ -1 \end{smallmatrix} \right)$ constructs a variety isomorphic to $T'$, where $T'$ has the same Cox ring of $T$ but its irrelevant ideal is $(y_1,y_2,t) \cap (x_1,x_2)$. Consider the map $f \colon T \rightarrow Q$ given by
\[
(y_1,y_2,t,x_1,x_2) \mapsto (t,y_1x_1,y_1x_2,y_2x_1,y_2x_2).
\] 
and similarly $g\colon T' \rightarrow Q$. Then $f$ contracts the locus $C^{-} \colon \mathbb{P}^1 \simeq (x_1=x_2=0) \subset T$ to $p \in Q$ and $g$ contracts $C^{+} \colon \mathbb{P}^1 \simeq (y_1=y_2=0) \subset T'$. Notice that $f$ and $g$ are isomorphisms away from the contracted loci. Hence, the map $\tau \colon T \rat T'$ replacing $C^-$ by $C^+$ is a small $\mathbb{Q}$-factorial modification that we denote by $(-1,-1,1,1)$. In fact, it is the Atyiah flop since $K_T \cdot C^- = K_{T'} \cdot C^+ = 0$. We say that $\tau$ is terminal since $T'$ is terminal (smooth, in fact). Moreover, the map 
\[
T' \rightarrow \mathbb{P}^1=\Proj \bigoplus_{m\geq 0}H^0(T',\mathcal{O}(m,-m)), \quad (y_1,y_2,t,x_1,x_2) \mapsto (x_1,x_2)
\]
is a fibration whose fibres are isomorphic to $\mathbb{P}^2$.
\end{Ex}

\begin{Lem} \label{lem:easycase}
Let $\varphi \colon T \rightarrow \mathbb{P}^d$ be the toric  $(1,1,\ldots,1,b)$-weighted blowup of $\mathbf{p}$. Then $-K_T$ is nef if and only if $b \in \{1,2,3\}$. The morpshism $\varphi$ initiates a toric Sarkisov link if and only if $b\in \{1,2\}$. Moreover, the hyperplanes passing through $\mathbf{p}$ induce a conic bundle structure on $T$ if $b=1$ or a divisorial contraction to $\mathbb{P}^{d-2} \subset \mathbb{P}^d(1,\ldots,1,2)$ if $b=2$.
\end{Lem}

\begin{proof}
The first assertion follows immediately from theorem \ref{thm:weak}. Equality happens only in the case $b=3$.  Let $T$ be the ordinary blowup of $\mathbb{P}^{d}$ at a coordinate point. Then,
\[
\Nef(T) = \Mov(T)=\mathbb{R}_+[H]+\mathbb{R}_+[H-E] \subset  \mathbb{R}_+[E]+\mathbb{R}_+[H-E] = \Eff(T)
\]
and $H-E$ is not big. The variety $T$ has a conic bundle structure induced by the hyperplanes passing through $p$. Indeed, it is given by the fibration 
\begin{align} \label{eq:fib}
\begin{split}
        \varphi' \colon T &\longrightarrow \Proj \bigoplus_{m\geq 0}H^0(T,m(H-E))\simeq \mathbb{P}^{d-1}\\
    (u,x_0,\ldots,x_{d}) & \longmapsto (x_1\colon\ldots\colon x_{d}).
    \end{split}
\end{align}
On the other hand $H-E$ is effective but not big provided that $b>1$. Indeed,
\[
\Nef(T) = \Mov(T)=\mathbb{R}_+[H]+\mathbb{R}_+[H-E] \subset  \mathbb{R}_+[E]+\mathbb{R}_+[H-bE] = \Eff(T)
\]
By theorem \ref{thm:weak}, $-K_T$ is in the interior of $\Nef(T)$ if and only if $b<3$. Hence $b=2$. Consider the map given by the linear system of multiples of $H-E$:
\begin{align*}
    \varphi' \colon T &\longrightarrow \Proj \bigoplus_{m\geq 0}H^0(T,m(H-E))\simeq \mathbb{P}^{d}(1,2,1,\ldots,1)\\
    (u,x_0,\ldots,x_{d}) & \longmapsto (ux_d : x_0x_d: x_1: \ldots : x_{d-1}).
\end{align*}
Then, $\varphi'$ contracts the divisor $\Bl_p \mathbb{P}^{d-1} \subset T$ to $\mathbb{P}^{d-2}$ and the restriction of $\varphi'$ to $\Bl_p \mathbb{P}^{d-1}$ is exactly the fibration \eqref{eq:fib} restricted to  $x_d=0$.
\end{proof}

\subsection{Sarkisov links from $\mathbb{P}^3$}

\begin{Thm} \label{thm:P3}
Let $\varphi \colon T \rightarrow \mathbb{P}^3$ be the toric $(1,a,b)$-weighted blowup of a point. Then $\varphi$ initiates a toric Sarkisov link from $\mathbb{P}^3$ if and only if 
\[
(a,b) \in \{(1,1), (1,2), (2,3), (2,5) \}.
\]
\end{Thm}

\begin{proof}
By Lemma \ref{lem:easycase}, we can assume $1<a<b$. In particular, $b>2$. In this case $T$ is a rank 2 toric variety with weight system
\begin{equation} \label{eq:P3}
\begin{array}{cccc|ccccc}
             &       & u     & x & y & z & v & \\
\actL{T}   &  \lBr &  0 & 1 & 1 & 1 & 1 &   \actR{.}\\
             &       & -1 & 0 & 1 & a & b &  
\end{array}
\end{equation}
Since $1<a<b$ we have
\[
\Mov(T)=\mathbb{R}_+[H]+\mathbb{R}_+[H-aE] \subsetneq  \mathbb{R}_+[E]+ \mathbb{R}_+[H-bE]=\Eff(T).
\]
\begin{figure}[h]
\centering
\begin{tikzpicture}[scale=2]
  \coordinate (A) at (0, 0);
  \coordinate [label={left:$E$}] (E) at (0, -1);
  \coordinate [label={right:$H$}] (K) at (1, 0);
\coordinate [label={right:$H-E$}] (y) at (1,1);
\coordinate [label={right:$H-aE$}] (M2) at (0.4,1.2);
\coordinate [label={above:$H-bE$}] (E') at (0.2,1.3);
\draw[fill=gray!30,draw=none]    (0,0) -- ++(1,0) -- ++(0,1) ;
    \draw  (A) -- (E);
    \draw  (A) -- (y);
    \draw [very thick,color=red] (A) -- (K);
    \draw (0.7,0.35) node{ \tiny $\Nef(T)$} ;
    \draw [very thick,color=red] (A) -- (M2);
	\draw (A) -- (E');
\end{tikzpicture}
\caption{A representation of the Mori chamber decomposition of $T$. The outermost rays generate the cone of pseudo-effective divisors of $T$ and in red it is represented the subcone of movable divisors of $T$.}%
\label{fig:chambdecomplin2}%
\end{figure}
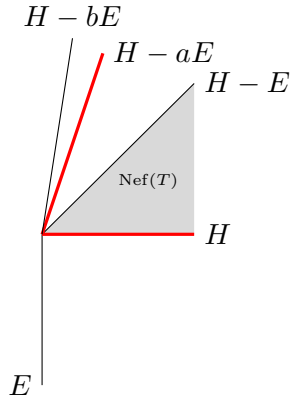 For the Sarkisov link to exist the class of the anticanonical divisor of $T$ must be in the interior of $\Mov(T)$. Since 
\[
-aK_T \sim  (3a-b) H + (a+b)(H-aE)
\]
it follows that $-K_T$ is in the interior of $\Mov(T)$ if and only if $3a>b$.


Let $D'$ be a movable divisor in the interior of the cone generated by the rays $\mathbb{R}_+[H-E] + \mathbb{R}_+[H-aE]$ and suppose $T'$ is the ample model for $D'$, that is, $T' \simeq \Proj R(T,D')$. Then $\Cox(T) = \Cox(T')$ but the irrelevant ideal of $T'$ is $(u,x,y) \cap (z,v)$. By \cite{mdsGIT} there is a small $\mathbb{Q}$-factorial modification $\tau \colon T \dashrightarrow T'$ which we describe explicitly. Let $\mathcal{F}_y$ be the ample model of the divisor $H-E$. Then,
\[
\mathcal{F}_y \simeq \Proj \bigoplus_{m\geq 0} H^0(T,m(H-E)) = \mathbb{C}[y, \ldots, u_i, \ldots ]
\]
where each $u_k$ is a monomial in the ideal $(u,x) \cap (z,v)$. Let $\alpha \colon T \rightarrow \mathcal{F}_y$ be the map given by $(u,x,y,z,v) \mapsto (y, \ldots, u_i, \ldots)$ and define similarly $\alpha' \colon T' \rightarrow \mathcal{F}_y$. Let $p_y \in \mathcal{F}_y$ be the point $(1:0: \ldots :0)$. It is clear that $\alpha$ contracts the locus $L \colon (z=v=0) \subset T$ to $p_y$ and that $\alpha'$ contracts $L' \colon (u=x=0) \subset T'$. Moreover, $L \simeq \mathbb{P}^1$ and $L' \simeq \mathbb{P}(a-1,b-1)$ are contracted to the same point. The small modification $\tau$ replaces $L$ with $L'$ and is denoted by
\[
(1,1,1-a,1-b).
\]

We find conditions on $a$ and $b$ for $T'$ to be terminal. 
Since 3-dimensional terminal singularities are isolated we have $\gcd(a-1,b-1)=1$. These are then the points
\[
P_1 = \frac{1}{a-1}(-1,-1,b-1)\,\, \text{and}\,\, P_2 = \frac{1}{b-1}(-1,-1,a-1)
\]
and both need to be terminal. The point $P_1$ is a terminal singularity if and only if 
\begin{enumerate}
    \item $-2 \equiv 0 \mod{a-1}$ and $\gcd(a-1,b-1)=1$ or
    \item $b \equiv 2 \mod{a-1}$
\end{enumerate}
\paragraph{Case 1.} We have that $a = 2$ or $a=3$. If $a=2$, then $P_1$ is smooth and $P_2=\frac{1}{b-1}(-1,-1,1)$ is terminal for every $b\geq 2$. However, $-K_T \in \Int\Mov(T)$ if and only if $a<b<3a$. Hence $b \in \{3,5 \}$. Notice that if $(a,b)=(2,3)$ the modification $\tau$ is Francia's antiflip. If $a=3$, the point $P_1$ is terminal if and only if $b$ is even. However, the point $P_2 = \frac{1}{b-1}(-1,-1,2)$ is terminal only if $b=2$ (in which case $P_2$ is actually smooth) which contradicts the assumption that $b>a$.

\paragraph{Case 2.} We have $a-1 | b-2$. Since $b<3a$, there is a positive integer $m$ for which 
\[
m(a-1)=b-2 < 3a-2.
\]
 Hence $m \in \{1,2,3\}$. If $m=1$, then $b=a+1$ and the point $P_2=\frac{1}{a}(-1,-1,a-1)\sim \frac{1}{a}(a-1,a-1,a-1)$ is terminal if and only if $a=2$ in which case  $b = 3$. If $m=2$, then $b=2a$ and $\gcd(a,b)=1$ if and only if $a=1$ which contradicts the assumption that $a>1$. If $m=3$, then $b=3a-1$ and
 \[
 P_2=\frac{1}{3a-2}(-1,-1,a-1)\sim \frac{1}{3a-2}(3(a-1),3(a-1),a-1).
 \]
 But $3(a-1) \equiv 1 \bmod 3a-2$ and so
 \[
 P_2\sim \frac{1}{3a-2}(3,3,1)
 \]
 by changing the generator $\epsilon$ of $\bm{\mu}_{3b-2}$ by the automorphism $\epsilon \mapsto \epsilon^{-3}$. By Theorem \ref{thm:YPG}, we have $1<3a-2<7$. Hence $a=2$ and $b=5$.

 \paragraph{} To conclude, the only choices of $(a,b)$ for which both $T$ and $T'$ are terminal and $-K_T \in \Int\Mov(T)$ are 
 \[
 (a,b) \in \{(1,1), (1,2), (2,3), (2,5) \}.
 \]
 By \cite[Lemma~2.9]{hamidquartic}, the map $\varphi'$ initiates a Sarkisov link precisely for these values of $a$ and $b$. 
\end{proof}

\begin{Cor}
Suppose $\varphi$ is the toric $(1,a,b)$-weighted blowup of $\mathbb{P}^3$ at a coordinate point where $(a,b) \in \{ (2,3), (2,5)\}$. Then $\varphi$ initiates a toric Sarkisov link ending with a divisorial contraction to a point in the terminal weighted projective space $\mathbb{P}(1,b,b-1,b-a)$.
\end{Cor}

\begin{proof}

Consider the map $\varphi' \colon T' \rightarrow \mathcal{F}$ given by the sections $\bigoplus_{m\geq 0}H^0(T',m(H-aE))$. The 3-fold $\mathcal{F}$ is isomorphic to the image of $\mathbb{P}^3(1,b,b-1,b-a)$ via the Veronese embedding given by the complete linear system $|\mathcal{O}(b-a)|$. We write $\varphi'$ explicitly.  Recall that $T'$ is the ample model of a divisor in the interior of the cone $\mathbb{R}_+[H-E]+\mathbb{R}_+[H-aE]$. Let
\[
A = \begin{pmatrix}
  a-1 & 1\\ 
  b-1 & 1
\end{pmatrix} \in \GL(2,\mathbb{Z}). 
\]
Multiplying the weight system of $T'$ on the left by $A$ yields an isomorphic toric variety to $T'$ given by
\begin{equation*}
\begin{array}{ccccc|cccc}
             &       & u     & x & y & z & v & \\
\actL{T'}   &  \lBr &  1 & a & a-1 & 0 & -(b-a) &   \actR{.}\\
             &       & 1 & b & b-1 & b-a & 0 &  
\end{array}
\end{equation*}
The hyperplane section $E'  := T'|_{v=0}$ is isomorphic to the weighted plane $\mathbb{P}(1,a,a-1)$. The map $\varphi'$ given by the sections with are multiples of the big divisor $H-aE$ is 
\begin{align*}
    \varphi' \colon T' &\rightarrow \mathbb{P}(1,b,b-1,b-a) \\
    (u,x,y,z,v) &\mapsto (uv^{\frac{1}{b-a}},xv^{\frac{a}{b-a}},yv^{\frac{a-1}{b-a}},z).
\end{align*}
The map $\varphi'$ is a divisorial contraction to the point $\mathbf{p}=(0:0:0:1) \in \mathbb{P}(1,b,b-1,b-a)$. If $(a,b)=(2,3)$, the map $\varphi'$ is a $(1,2,1)$-Kawakita blowup of the smooth point $\mathbf{p}$ with exceptional divisor isomorphic to $\mathbb{P}(1,1,2)$. If $(a,b)=(2,5)$, the map $\varphi'$ is a $\frac{1}{3}(1,2,1)$-Kawamata blowup of the terminal cyclic quotient singularity $\mathbf{p} \sim \frac{1}{3}(1,2,1)$ with exceptional divisor isomorphic to $\mathbb{P}(1,1,2)$.
\end{proof}

\begin{table}[ht!]
    \centering
    \begin{tabular}[t]{l c c c}
        \toprule
        $(a,b)$  & $\tau$ & $\varphi'$             &  Model     \\ \midrule

 $(1,1)$ &    & Fibration &   $\mathbb{P}^1$-bundle over $\mathbb{P}^2$   \\

 $(1,2)$ &    & Divisorial Contraction to $\mathbb{P}^1$ & $\mathbb{P}(1,1,1,2)$   \\

 $(2,3)$ &  $(1,1,-1,-2)$  & $(1,1,2)$-Weighted blowup of a smooth point &  $\mathbb{P}(1,1,2,3)$   \\
        
 $(2,5)$ &  $(1,1,-1,-4)$  & Kawamata blowup of $\frac{1}{3}(1,1,2)$  &  $\mathbb{P}(1,3,4,5)$  \\

        \bottomrule
    \end{tabular}
\caption{Table summarising the results obtained. The first column denotes the weights of the $(1,a,b)$-Kawakita blowup of a coordinate point. The second column denotes a terminal small $\mathbb{Q}$-factorial modification. The third column is the last birational morphism and the last column denotes the new model of $\mathbb{P}^3$. }
\label{tab:fibrations}
\end{table}

\subsection{Sarkisov links from $\mathbb{P}^4$}

We look at the same problem for extractions from a point in $\mathbb{P}^4$. This is much more complicated than the previous case but our methods still allow for a definitive result. 

\begin{Lem} \label{lem:auxP41}
Let $\varphi \colon T \rightarrow \mathbb{P}^4$ be the toric $(1,1,1,d)$-weighted blowup of a point. Then $\varphi$ initiates a toric Sarkisov link from $\mathbb{P}^4$ if and only if $d\in \{1,2\}$.
\end{Lem}

\begin{proof}  See Lemma \ref{lem:easycase}.
\end{proof}

\begin{Lem} \label{lem:auxP42}
Let $\varphi \colon T \rightarrow \mathbb{P}^4$ be the toric $(1,1,c,d)$-weighted blowup of a point with $1<c\leq d$. Then $\varphi$ initiates a toric Sarkisov link from $\mathbb{P}^4$ if and only if $(c,d)\in \{(2,d)\, |\, 2 \leq d \leq 6 \}$.
\end{Lem}

\begin{proof}
We have $-K_T \sim 5H - (d+c+1)E$ and divide the proof in two subcases:
\paragraph{Suppose $1<c=d$.}
The Movable cone of $T$ is subdivided in two chambers. Let $T$ and $T'$ be the ample models of each chamber. Then, there is a small $\mathbb{Q}$-factorial modification $\tau$ between $T$ and $T'$. The weight system of $T$ is
\begin{equation*}
\begin{array}{cccc|ccccc}
             &       & u     & x & y & z & t & v & \\
\actL{T}   &  \lBr &  0 & 1 & 1 & 1 & 1 & 1 &  \actR{.}\\
             &       & -1 & 0 & 1 & 1 & d & d &
\end{array}
\end{equation*}
Let $\mathcal{F}=\Proj \bigoplus_{m\geq0}H^0(T,\mathcal{O}(m(H-E)))=\Proj\bigoplus_{m\geq0}H^0(T',\mathcal{O}(m(H-E)))$ and $\alpha \colon T \rightarrow \mathcal{F},\,\ \beta \colon T' \rightarrow \mathcal{F}$ the associated contractions. The exceptional locus of $\alpha$ is $T|_{t=v=0}$ which is isomorphic to $\mathbb{F}_1$. On the other hand the exceptional locus of $\beta$ is $S:=T'|_{u=x=0}$. Notice that if $d>2$,  $T'$ contains a surface of singularities. Hence $d=2$, since we want $T'$ to be terminal. Therefore $S \simeq \mathbb{P}^1\times \mathbb{P}^1$.

We conclude that the map $\tau$ fits into the diagram 
\begin{equation*}
    \xymatrix{
    \mathbb{F}_1 \subset T \ar@{-->}[rr]^\tau \ar[dr]_{\alpha} && T' \supset \mathbb{P}^1\times\mathbb{P}^1 \ar[dl]^{\beta} \\
    & \mathbb{P}^1\subset \mathcal{F} &
    }
\end{equation*}
where $\alpha$ and $\beta$ are the small contraction
\[
(u,x,y,z,t,v) \mapsto (y,z,ut,uv,xt,xv).
\]
 Hence, the map $\tau$ swaps the two degree 8 del Pezzo surfaces: the Hirzebruch surfaces $\mathbb{F}_1$ and $\mathbb{P}^1\times \mathbb{P}^1$. This is a 4-dimensional analog to the Atyiah flop where, over each point of the base, $\mathbb{P}^1\subset \mathcal{F}$, we have a 3-fold Atyiah flop. Indeed, the 4-fold $\mathcal{F}$ is isomorphic to
\[
(u_1u_4-u_2u_3=0) \subset \mathbb{P}(1,1,1,1,2,2)
\]
with homogeneous variables $y,\,z,\,u_1,\,u_2,\,u_3,\,u_4$. Finally the map 
\begin{align*}
\varphi' \colon T' &\longrightarrow \Proj \bigoplus_{m\geq 0} H^0(T',m(H-2E)) \simeq \mathbb{P}^1   \\
(u,x,y,z,t,v) &\longmapsto (t: v)
\end{align*}
is a fibration whose fibres are isomorphic to $\mathbb{P}(1,1,1,2)$. Notice that $H-2E$ is not big.

\paragraph{Suppose $1<c<d$.} We have $-K_T \in \Int\Mov(T)$ if and only if $4c>d+1$. Let $T'$ be the ample model in the Mori chamber adjacent to $\Nef(T)$. Then $T$ and $T'$ are related by a small $\mathbb{Q}$-factorial modification $\tau$. The surfaces $S:=T|_{t=v=0}=\mathbb{F}_1$ and 
\begin{equation*}
\begin{array}{cccc|ccccc}
           &       & y     & z & t & v  & \\
\actL{S'}   &  \lBr &  1 & 1 & 1 & 1  &  \actR{}\\
           &       & 0 & 0 & 1-c & 1-d & 
\end{array}
\end{equation*}
 are both $\mathbb{P}^1$-bundles over $\mathbb{P}^1$ and $\tau|_S$ is the composition of these fibrations. Moreover $\tau$ is a fibre-wise small modification, that is, for each point $p \in \mathbb{P}^1$, the fibre $\alpha^{-1}(p) = \mathbb{P}^1$ in $S$ is replaced with the corresponding fibre $\beta^{-1}(p)=\mathbb{P}(c-1,d-1)$ in $S'$. 
 
 Suppose that $c=2$. In that case $\tau$ introduces the line of singularities 
 \[
 \Gamma \sim\frac{1}{d-1}(-1,-1,1) \subset S'
 \]
 which is clearly terminal for any $d>1$. Suppose on the other hand that $c>2$. We claim that $T'$ is not terminal for any $d>c$. We use the observation that $\tau$ is a fibre-wise small modification. Over each point of the base of $\tau$ we extract the line of singularities
  \[
 \Gamma \sim\frac{1}{d-1}(-1,-1,c-1) \subset S'
 \]
 and $\Gamma$ is terminal if and only if
\begin{itemize}
    \item $-2 \equiv 0 \bmod{d-1}$ (and $\gcd(b-1,c-1)=1)$ or
    \item $c-2 \equiv 0 \bmod{d-1}$.
\end{itemize}
For the first case $d-1 | 2$ which implies that $d \in \{2,3\}$. This is not possible since we assumed $d>c>2$. For the second case, notice that there is $m \in \mathbb{Z}_+$ for which $m(d-1)=c-2 < d-2$ since $c<d$ which is impossible. Hence, if $c>2$, the modification $\tau$ introduces a line of non-terminal singularities over each point of the base. Hence $S'$ is not terminal and we conclude that the small modification $T \dashrightarrow T'$ is terminal if and only if $c=2$ as we claimed.

Since $c=2$, it follows that $d \in \{3, 4, 5, 6 \}$.  Consider the map $\varphi' \colon T' \rightarrow \mathcal{F}$ where
\[
\mathcal{F}=\Proj \bigoplus_{m\geq 0} H^0(T',m(H-2E)) \simeq \mathbb{P}(1,d,d-1,d-1,d-2).
\]
The map $\varphi'$ is a divisorial contraction to the point $\mathbf{p}'=(0:0:0:0:1) \in \mathbb{P}(1,d,d-1,d-1,d-2)$. The point $\mathbf{p}'$ is smooth if and only if $d=3$ and is a terminal cyclic quotient singularity of type $\frac{1}{d-2}(1,d,d-1,d-1)$ otherwise. Notice that in this case, $H-2E$ is a big divisor.

\end{proof}

\begin{Lem} \label{lem:auxP43}
Let $\varphi \colon T \rightarrow \mathbb{P}^4$ be the toric $(1,b,c,d)$-weighted blowup of a point with $1<b\leq c\leq d$, where one of the inequalities is an equality. Then $\varphi$ initiates a toric Sarkisov link from $\mathbb{P}^4$ if and only if 
\[
(b,c,d)\in \{(2,3,3),(2,5,5), (2,2,3), (2,2,5),(3,3,4), (3,3,8)\}.
\]
\end{Lem}

\begin{proof}
Since $T$ is terminal it follows that $\gcd(b,c,d)=1$. Otherwise the exceptional divisor contains a surface of singularities. In particular at least one of the inequalities in $b\leq c \leq d$ is strict.

We have $-K_T = 5H-(b+c+d)E$.

\paragraph{Suppose $1<b<c=d$.} The following holds
\[
\Mov(T)= \mathbb{R}_+[H]+\mathbb{R}_+[H-cE] \subsetneq \mathbb{R}_+[E]+\mathbb{R}_+[H-cE] = \Eff(T). 
\]

The anticanonical divisor of $T$ is in the interior of $\Mov(T)$ since $-K_T = 2H+2(H-cE)+H-bE$ where $H-bE$ is in the interior of $\Mov(T)$ by the assumpton $b<c$. The movable cone of $T$ is subdivided into three chambers.  
\begin{align*}
    \Nef(T) &= \mathbb{R}_+[H]+\mathbb{R}_+[H-E] \\
    \Nef(T_1) &= \mathbb{R}_+[H-E]+\mathbb{R}_+[H-bE]\\
    \Nef(T_2) &= \mathbb{R}_+[H-bE]+\mathbb{R}_+[H-cE].
\end{align*}

There is a sequence of small $\mathbb{Q}$-factorial modifications $T \dashrightarrow T_1 \dashrightarrow T_2$. The contraction $\alpha \colon T \rightarrow \mathcal{F}$ where $\mathcal{F}=\Proj \bigoplus_{m\geq 0} H^0(T,m(H-E))$ has exceptional locus isomorphic to $\mathbb{P}^1$ and contracts it to a point in $\mathcal{F}$. On the other side we have $\alpha' \colon T' \rightarrow \mathcal{F}$ whose exceptional locus is isomorphic to $\mathbb{P}(b-1,c-1,c-1)$.  We claim that $T'$ is terminal if and only if $b=2$. Notice that $\gcd(b-1,c-1)=1$ and $T'$ has an extra curve of singularities $\Gamma$ of type $\frac{1}{c-1}(-1,-1,b-1)$. Hence, if $b=2$, the map $T \dashrightarrow T'$ adds a curve of singularities which is of type $\frac{1}{c-1}(-1,-1,1)$ and therefore $T'$ is terminal. On the other hand, suppose $T'$ is terminal. Notice that $\Gamma$ is terminal if and only if 
\begin{itemize}
    \item $-2 \equiv 0 \pmod{c-1}$ (and $\gcd(b-1,c-1)=1)$ or
    \item $b-2 \equiv 0 \pmod{c-1}$.
\end{itemize}
For the first case $c-1 | 2$ which implies that $c \in \{2,3\}$. We assumed $c>b>1$ and so $c=3$ and $b=2$. For the second case, notice that there is $m \in \mathbb{Z}_+$ for which $m(c-1)=b-2 < c-2$ since $b<c$ which is impossible. We conclude that the small modification $T \dashrightarrow T'$ is terminal if and only if $b=2$ as we claimed.

We now proceed to the second small modification modification $T' \dashrightarrow T''$. We consider 
\[
\mathcal{F}'=\Proj \bigoplus_{m\geq 0}H^0(T',m(H-bE))
\]
and $\beta \colon T' \rightarrow \mathcal{F}',\, \beta' \colon T'' \rightarrow \mathcal{F}'$ given by the sections of multiples of the movable divisor $H-bE$. Then, $\beta$ contracts the surface $\mathbb{P}(1,1,2)$ to a point in $\mathcal{F}$ and $\beta'$ extracts the line $\mathbb{P}(c-2,c-2)$ introducing the line of singularities 
\[
L \sim \frac{1}{c-2}(-1,-2,-1)
\]
in $T''$. By terminality, it follows that $c-2\, |\, -2$ and $\gcd(2,c-2)=1$ or $c-2\, |\, -3 $. In the first case, $c =3$ and $L$ is smooth in $T''$. In the second case, $c \in \{3,5\}$ and $L \sim \frac{1}{3}(1,1,2)$. We conclude that if $1<b<c=d$ the 4-folds $T,\, T'$ and $T''$ are terminal if and only if 
\[
(b,c,d) \in \{(2,3,3),(2,5,5) \}.
\]

\paragraph{Suppose $1<b=c<d$.} We have
\[
\Mov(T)= \mathbb{R}_+[H]+\mathbb{R}_+[H-bE] \subsetneq \mathbb{R}_+[E]+\mathbb{R}_+[H-dE] = \Eff(T)
\]
and 
\[
-bK_T \sim (3b-d)H+(2b+d)(H-bE).
\]
Hence, it is in the interior of $\Mov(T)$ if and only if $d < 3b$. The movable cone of $T$ is subdivided into two chambers.  
\begin{align*}
    \Nef(T) &= \mathbb{R}_+[H]+\mathbb{R}_+[H-E] \\
    \Nef(T') &= \mathbb{R}_+[H-E]+\mathbb{R}_+[H-bE].
\end{align*}
Hence there is a small modification $\tau \colon T \dashrightarrow T'$ that we determine: Consider the movable divisor $H-E$ and $\mathcal{F}=\Proj \bigoplus_{m\geq 0}H^0(T,m(H-E))$. Let $\alpha \colon T \rightarrow \mathcal{F}$ and $\alpha' \colon T' \rightarrow \mathcal{F}$ be the associated contractions. The exceptional loci of $\alpha$ and $\alpha'$ are $\mathbb{P}^1$ and $\mathbb{P}(b-1,b-1,d-1)$, respectively. Hence $\gcd(b-1,d-1)=1$. Therefore $\tau$ introduces a curve of singularities 
\[
\Gamma \sim \frac{1}{b-1}(-1,-1,d-1)
\]
and a point $p \sim \frac{1}{d-1}(-1,-1,b-1,b-1)$.
The curve $\Gamma$ is terminal if and only if 
\begin{enumerate}
    \item $-2 \equiv 0 \pmod{b-1}$ (and $\gcd(b-1,d-1)=1)$ or 
    \item $d-2 \equiv 0 \pmod{b-1}$.
\end{enumerate}

\hspace{0.5cm} \textbf{Case 1.} We have $b \in \{2, 3 \}$. Since $d<3b$ we have the following possibilities
\[
(b,d) \in \{(2,3), (2,5), (3,4), (3,8) \}.
\]
The point $p$ is terminal in each of these cases.

\hspace{0.5cm} \textbf{Case 2.} By assumption there is a positive integer $m$ for which
\[
m(b-1)=d-2<3b-2.
\]
Therefore $m \in \{1,2,3 \}$. If $m=1$, then $b=d-1$ and $\Gamma$ is terminal for every $d$. However the point $p \sim \frac{1}{d-1}(1,1,1,1)$, where $d>2$ is terminal if and only if $d \in \{3,4 \}$. The possibilities are therefore
\[
(b,d) \in \{(2,3),(3,4) \}.
\]
If $m=2$, then $d=2b$ and $\gcd(b,c,d) > 1$. 
If $m=3$, then $d=3b-1$. The curve $\Gamma$ is terminal. On the other hand,
\[
p \sim \frac{1}{3b-2}(-1,-1,b-1,b-1) \sim \frac{1}{3b-2}(3(b-1),3(b-1),b-1,b-1).
\]
But $(3b-5)(b-1) \equiv 1 \bmod{3b-2}$ and so,
\[
p \sim \frac{1}{3b-2}(3,3,1,1)
\]
by changing the generator $\epsilon$ of $\bm{\mu}_{3b-2}$ by the automorphism $\epsilon \mapsto \epsilon^{3b-5}$.
Now it is clear that $1< 3b-2 <  8$, that is, $b \in \{2,3 \}$. The possibilities are
\[
(b,d) \in \{(2,5), (3,8) \}.
\]
 We conclude that if $1<b=c<d$ the 4-folds $T$ and $T'$ are terminal if and only if 
\[
(b,c,d) \in \{(2,2,3),(2,2,5), (3,3,4), (3,3,8) \}.
\]
\end{proof}

\begin{Lem} \label{lem:auxP44}
Let $\varphi \colon T \rightarrow \mathbb{P}^4$ be the toric $(a,b,c,d)$-weighted blowup of a point with $1<a\leq b< c= d$. Then $\varphi$ initiates a toric Sarkisov link from $\mathbb{P}^4$ if and only if 
\[
(a,b,c,d)\in \{(2,3,5,5)\}.
\]
Moreover, the Sarkisov link ends with a fibration to $\mathbb{P}^1$ whose fibres are isomorphic to $\mathbb{P}(1,2,3,5)$.
\end{Lem}

\begin{proof}
Since $T$ is terminal it follows that $\gcd(b,c)=1$. Otherwise the exceptional divisor contains a surface of singularities. We have $-K_T = 5H-(a+b+c+d-1)E$.

\paragraph{Suppose $1<a=b<c=d$.} The following holds
\[
\Mov(T)= \mathbb{R}_+[H]+\mathbb{R}_+[H-cE] \subsetneq \mathbb{R}_+[E]+\mathbb{R}_+[H-cE] = \Eff(T). 
\]

The anticanonical divisor of $T$ is in the interior of $\Mov(T)$ since $-cK_T = (3c-2a+1)H+(2a+2c-1)(H-cE)$ and $a<c$ by assumption. The movable cone of $T$ is subdivided into two chambers.  
\begin{align*}
    \Nef(T) &= \mathbb{R}_+[H]+\mathbb{R}_+[H-aE] \\
    \Nef(T_1) &= \mathbb{R}_+[H-aE]+\mathbb{R}_+[H-cE]
    \end{align*}
Hence, there is a small $\mathbb{Q}$-factorial modification $\tau \colon T \dashrightarrow T_1$. The divisor $H-cE$ is in the boundary of $\Eff(T)$ and so it is not big. Let $\mathcal{F}=\Proj \bigoplus_{m\geq 0}H^0(T,m(H-cE)) = \mathbb{P}^1$ and $\varphi' \colon T \rightarrow \mathbb{P}^1$ be the associated contraction. This is a fibration whose fibres are isomorphic to $\mathbb{P}(1,c,c-a,c-a)$. By terminality we have $a=1$ and $c=2$ contradicting the assumption that $a>1$. See lemma \ref{lem:auxP42}, where the case of a $(1,1,2,2)$-weighted blowup is treated. 

\paragraph{Suppose $1<a<b<c=d$.} The following holds
\[
\Mov(T)= \mathbb{R}_+[H]+\mathbb{R}_+[H-cE] \subsetneq \mathbb{R}_+[E]+\mathbb{R}_+[H-cE] = \Eff(T). 
\]

The anticanonical divisor of $T$ is in the interior of $\Mov(T)$ if and only if $3c>a+b-1$ since $-cK_T = (3c-a-b+1)H+(a+b+2c-1)(H-cE)$. The movable cone of $T$ is subdivided into three chambers.  
\begin{align*}
    \Nef(T) &= \mathbb{R}_+[H]+\mathbb{R}_+[H-aE] \\
    \Nef(T_1) &= \mathbb{R}_+[H-aE]+\mathbb{R}_+[H-bE]\\
    \Nef(T_2) &= \mathbb{R}_+[H-bE]+\mathbb{R}_+[H-cE]
    \end{align*}
Hence, there is a small $\mathbb{Q}$-factorial modification $\tau \colon T \dashrightarrow T_1 \dashrightarrow T_2$. The divisor $H-cE$ is in the boundary of $\Eff(T)$ and so it is not big. Let $\mathcal{F}=\Proj \bigoplus_{m\geq 0}H^0(T,m(H-cE)) = \mathbb{P}^1$ and $\varphi' \colon T \rightarrow \mathbb{P}^1$ be the associated contraction. This is a fibration whose fibres are isomorphic to $\mathbb{P}(1,d-b,d-a,d)$. By terminality and the assumption that $a>1$ we have $a=2,\,b=3,\,c=5,\,d=5$. 
\end{proof}

\paragraph{The remaining cases.} Given the vast number of possibilities for the weights of the weighted blowup of $\mathbb{P}^4$, we end the classification with the help of the computer. To be clear the remaining cases are the weighted blowups of weights $(a,b,c,d)$ where 
\begin{enumerate}
    \item $1<a\leq b \leq c < d$, with exactly one equality.
    \item $0<a<b<c<d$. 
\end{enumerate}
The main idea of the proof is the use of the classification of terminal $\mathbb{Q}$-factorial Fano 4-fold weighted projective spaces due to Kasprzyk, in \cite[Theorem~3.5]{WPS4}. A much more general statement is due to Birkar and asserts that for every $d \in \mathbb{N}$, the family of $d$-dimensional Fano varieties with terminal singularities is bounded. See \cite{BirkarI,BirkarII}. 

If the toric weighted blowup of a point in $\mathbb{P}^4$ initiates a Sarkisov link to a Fano 4-fold, then it must be a weighted projective space. More concretely, if the toric $(a,b,c,d)$-weighted blowup initiates a Sarkisov link to a Fano 4-fold $X'$, then $X'$ is isomorphic to $\mathbb{P}(1,d-c,d-b,d-a,d)$.

\begin{Lem} \label{lem:auxP45}
Let $\varphi \colon T \rightarrow \mathbb{P}^4$ be the toric $(a,b,c,d)$-weighted blowup of a point with $1<a \leq  b\leq c< d$ with exactly one equality.  Then $\varphi$ initiates a toric Sarkisov link from $\mathbb{P}^4$ if and only if 
\[
(a,a,c,d) \in \{(2,2,3,5),(2,2,3,7),(3,3,4,5),(3,3,4,10),(4,4,5,7),(5,5,6,8) \}
\]
up to permutation.
\end{Lem}

\begin{proof}
We use the classification of $\mathbb{Q}$-factorial terminal weighted projective spaces of dimension four in \cite{WPS4}. As we observed, if the toric $(a,b,c,d)$-weighted blowup of a point in $\mathbb{P}^4$ initiates a Sarkisov link to $X$ then $X$ is isomorphic to $\mathbb{P}(1,d-c,d-b,d-a,d)$.
\paragraph{Suppose $1<a=b<c<d$.} 
We have $-K_T = 5H-(2a+c+d-1)E$. Moreover, $-K_T$ is in the interior of the movable cone of $T$ if and only if $4c>2a+d-1$ since
\[
\Mov(T)= \mathbb{R}_+[H]+\mathbb{R}_+[H-cE] \subsetneq \mathbb{R}_+[E]+\mathbb{R}_+[H-dE] = \Eff(T) 
\]
and $-cK_T=(4c-2a-d+1)H+(2a+c+d-1)(H-cE)$. The movable cone of $T$ is subdivided into two chambers.  
\begin{align*}
    \Nef(T) &= \mathbb{R}_+[H]+\mathbb{R}_+[H-aE] \\
    \Nef(T') &= \mathbb{R}_+[H-aE]+\mathbb{R}_+[H-cE]
\end{align*}
And let $\tau\colon T \dashrightarrow T_1$ be the induced small $\mathbb{Q}$-factorial modification. Then, $\tau$ fits into the diagram
\begin{equation*}
    \xymatrix{
    S \subset T \ar@{-->}[rr]^\tau \ar[dr]_{\alpha} && T' \supset S' \ar[dl]^{\beta} \\
    & \mathbb{P}^1\subset \mathcal{F} &
}
\end{equation*}
where $S$ and $S'$ are geometrically ruled surfaces contracted to $\mathbb{P}^1$. More concretely,
\begin{equation*}
\begin{array}{cccc|ccccc}
           &       & u     & x & y & z  & \\
\actL{S}   &  \lBr &  0 & 1 & 1 & 1  &  \actR{}\\
           &       & 1 & a & 0 & 0 & 
\end{array}
\quad
\begin{array}{cccc|ccccc}
           &       & y     & z & t & v  & \\
\actL{S'}   &  \lBr &  1 & 1 & 1 & 1  &  \actR{.}\\
           &       & 0 & 0 & a-c & a-d & 
\end{array}
\end{equation*}
Given a terminal $\mathbb{Q}$-factorial weighted projective space $\mathbb{P}(1,d-c,d-a,d-a,d)$, we consider its associated toric  $(a,a,c,d)$-weighted blowup  $\varphi \colon T \rightarrow \mathbb{P}^4$ and use Algorithm \ref{terminalBlowupQ} to understand which $\varphi$ are in the Mori category.  These are the following:
\[
(a,a,c,d) \in \{(2,2,3,5),(2,2,3,7),(3,3,4,5),(3,3,4,10),(4,4,5,7),(5,5,6,8) \}.
\]
For each of the six tuples it is easy to check terminality of $T'$. Hence, $T$ and $T'$ are terminal and we consider the map $\varphi' \colon T' \rightarrow \mathbb{P}(1,d-c,d-a,d-a,d)$ given by the sections which are multiplies of the big divisor $H-cE$. The map $\varphi$ is a divisorial contraction to a terminal cyclic quotient singularity of type $\frac{1}{d-c}(1,d,d-a,d-a)$.
\paragraph{Suppose $1<a<b=c<d$.} 
Given a terminal $\mathbb{Q}$-factorial terminal weighted projective space $\mathbb{P}(1,d-b,d-b,d-a,d)$ we consider its associated toric $(a,b,b,d)$-weighted blowup $\varphi \colon T \rightarrow \mathbb{P}^4$. By the assumption that $0<a<b=c<d$ we have
\[
\Mov(T)= \mathbb{R}_+[H]+\mathbb{R}_+[H-bE] \subsetneq \mathbb{R}_+[E]+\mathbb{R}_+[H-dE] = \Eff(T). 
\]

The anticanonical divisor of $T$ is in the interior of $\Mov(T)$ if and only if $3b>a+d-1$ since $-bK_T = (3b-a-d+1)H+(a+2b+d-1)(H-bE)$. The movable cone of $T$ is subdivided into two chambers:  
\begin{align*}
    \Nef(T) &= \mathbb{R}_+[H]+\mathbb{R}_+[H-aE] \\
    \Nef(T_1) &= \mathbb{R}_+[H-aE]+\mathbb{R}_+[H-bE].
    \end{align*}
Hence, there is a small $\mathbb{Q}$-factorial modification $\tau \colon T \dashrightarrow T_1$ which is the anti-flip
\[
(-1,-a,b-a,b-a,d-a).
\]

We use Algorithm \ref{terminalBlowupQ} to understand which are the $(a,b,b,d)$-weighted blowups in the Mori category. These are the following:
\[
(a,b,b,d) \in \{(2,3,3,4),(2,5,5,6)\}.
\]

For each of the two tuples it is easy to check terminality of $T'$. Hence, $T$ and $T'$ are terminal and we consider the map $\varphi' \colon T' \rightarrow \mathbb{P}(1,d,d-a,d-b,d-b)$ given by the sections which are multiplies of the big divisor $H-bE$. The map $\varphi$ is a divisorial contraction centred at $\mathbb{P}^1 \subset \mathbb{P}(1,d,d-a,d-b,d-b)$.
\end{proof}

\begin{Lem} \label{lem:auxP46}
Let $\varphi \colon T \rightarrow \mathbb{P}^4$ be the toric $(a,b,c,d)$-weighted blowup of a point with $0<a < b< c< d$. Then $\varphi$ initiates a toric Sarkisov link from $\mathbb{P}^4$ if and only if $(a,b,c,d)$ is one of 399 quadruples up to permutation.
\end{Lem}

\begin{proof}
 As we observed, if the toric $(a,b,c,d)$-weighted blowup of a point in $\mathbb{P}^4$ initiates a Sarkisov link to $X$ then $X$ is isomorphic to $\mathbb{P}(1,d-c,d-b,d-a,d)$. By the assumption that $0<a<b<c<d$ we have
\[
\Mov(T)= \mathbb{R}_+[H]+\mathbb{R}_+[H-cE] \subsetneq \mathbb{R}_+[E]+\mathbb{R}_+[H-dE] = \Eff(T). 
\]

The anticanonical divisor of $T$ is in the interior of $\Mov(T)$ if and only if $4c>a+b+d-1$ since $-cK_T = (4c-a-b-d+1)H+(a+b+c+d-1)(H-cE)$. The movable cone of $T$ is subdivided into three chambers:  
\begin{align*}
    \Nef(T) &= \mathbb{R}_+[H]+\mathbb{R}_+[H-aE] \\
    \Nef(T_1) &= \mathbb{R}_+[H-aE]+\mathbb{R}_+[H-bE]\\
    \Nef(T_2) &= \mathbb{R}_+[H-bE]+\mathbb{R}_+[H-cE]
    \end{align*}
Hence, there is a small $\mathbb{Q}$-factorial modification $\tau \colon T \dashrightarrow T_2$ which can be decomposed as
\[
(-1,-a,b-a,c-a,d-a) \quad (-1,-b,a-b,c-b,d-b).
\]

Let $X=\mathbb{P}(1,d-c,d-b,d-a,d)$ and consider the corresponding toric  $(a,b,c,d)$-weighted blowup $\varphi \colon T \rightarrow \mathbb{P}^4$. We check whether $\varphi$ is an extraction in the Mori category with Algorithm \ref{terminalBlowupQ}. 
For each of these we check terminality of the singularities introduced by $\tau$ using Algorithm \ref{terminalWPSQ}.

This is the case for exactly 399 tuples. For each of those tuples, $T,\, T_1$ and $T_2$ are terminal $\mathbb{Q}$-factorial 4-folds and $\varphi' \colon T_2 \rightarrow X$ is a $K_{T_2}$-negative divisorial contraction. 
\end{proof}

The final result of our paper is the following theorem:

\begin{Thm} \label{thm:P4}
Let $\varphi \colon T \rightarrow \mathbb{P}^4$ be the toric $(a,b,c,d)$-weighted blowup of a point. Then $\varphi$ initiates a toric Sarkisov link from $\mathbb{P}^4$ if and only if $(a,b,c,d)$ is one of 421 quadruples up to permutation.
\end{Thm}

\section{Appendix: Code} \label{sec:code}

In this section we present the pseudo-code that checks terminality at each step of the Sarkisov program. We use them in the proof of Lemmas \ref{lem:auxP45} and \ref{lem:auxP46}. 

The auxiliary algorithm \ref{terminalSingQ} is an implementation of Theorem \ref{thm:terminalsing}.

\begin{algorithm} 
    \caption{Terminality of cyclic quotient singularity}\label{terminalSingQ}
    \hspace*{\algorithmicindent} \textbf{Input} \textit{List $L$; Integer $V$} 
    \begin{algorithmic}[1]
    \Procedure{terminalSingQ}{}
    \State $Q \gets true$
    \State $\textit{l} \gets \textit{length of L}$
    \State $\textit{sumtest} \gets \infty$
    \For {$k=1,\ldots, V-1$}
                \While{$sumtest > V$}
				\State  $sumtest \gets \sum_{i=1}^{l} (L[i]\cdot k \bmod V)$ 
				\If {$sumtest \leq  V$} 
				\State $Q \gets \textit{false}$ 
				\EndIf
				\EndWhile
			\EndFor
    \State $Q$;
    \EndProcedure
    \end{algorithmic}
    \end{algorithm}

\begin{Ex}
\begin{align*}
\text{terminalSingQ}([1,14,13,10],7)&=true\\ \text{terminalSingQ}([1,1,4,3],9)&=false\\\text{terminalSingQ}([-1,3,2],5)&=true; 
\end{align*}
\end{Ex}

Algorithm \ref{terminalBlowupQ} is an implementation of theorem \ref{thm:wtbup}.

\begin{algorithm}
    \caption{Terminality of Weighted blowup of a point}\label{terminalBlowupQ}
    \hspace*{\algorithmicindent} \textbf{Input} \textit{List $L$} 
    \begin{algorithmic}[1]
    \Procedure{terminalBlowupQ}{}
    \State $V  \gets -1 + \sum_{x \in L}x $
    \State $\textit{terminalSingQ(L,V)}$;
    \EndProcedure
    \end{algorithmic}
    \end{algorithm}
    
    \begin{Ex}
\begin{align*}
\text{terminalBlowupQ}([1,3,5])&=true\\ \text{terminalBlowupQ}([2,3,5])&=false\\\text{terminalBlowupQ}([2,3,6,7])&=true; 
\end{align*}
\end{Ex}

The auxiliary algorithm \ref{singList} gives the list of singularity indices of a weighted projective space.

\begin{algorithm}
    \caption{Singularity indices of a Weighted Projective Space}\label{singList}
    \hspace*{\algorithmicindent} \textbf{Input} \textit{List $L$} 
    \begin{algorithmic}[1]
    \Procedure{singList}{}
    \State $L_{aux} \gets \textit{empty list}$
    \State $L_{>1} \gets \textit{sublist of elements of L bigger than 1}$
    \State $L_p \gets \textit{powerset of}\,\,  L_{>1}$
    \State $k \gets 1$
                \While{$k \leq  \textit{length of}\,\, L_{p}$}
				\If {$gcd(L_{p}[k])>1$} 
				\State ${L_{aux} \gets Append\,\, gcd(L_{p}[k])\,\,  to\,\, L_{aux}}$  
				\EndIf
				$k \gets k+1$
				\EndWhile
   \State {$L_{aux}$;}
    \EndProcedure
    \end{algorithmic}
    \end{algorithm}

\begin{Ex}

The weighted projective space $\mathbb{P}(1,1,3,6,8)$ has the following basket of singularities
\[
\mathcal{B}=\bigg\{\frac{1}{6}(1,1,2,3),\frac{1}{8}(1,1,3,6), \frac{1}{2}(1,1,1),\frac{1}{3}(1,1,2)\bigg\}.
\]
The indices of the singularities are $\{2,3,6,8\}$. Indeed,
\[
\text{singList}([1,1,3,6,8])=[8,6,2,3].
\]
\end{Ex}

\begin{Ex}Notice the following behaviour:
\begin{align*}
\text{singList}([7,7,3,6,8])&=[7,8,6,2,3]\\
\text{singList}([-7,-7,3,6,8])&=[8,6,2,3].
\end{align*}
This is necessary because we want to check terminality of small $\mathbb{Q}$-factorial modifications. See Example \ref{ex:nontermflip}.
\end{Ex}

Algorithm \ref{terminalWPSQ} checks if a weighted projective space is terminal. More generally, given a list of integers $L$ it checks whether any \textit{positive} integer in $L$ is an index of a terminal cyclic quotient singularity with weights the elements of $L$. This is the core algorithm used to obtain Theorem \ref{thm:P4}.

\begin{algorithm}
    \caption{Terminality of a Weighted Projective Space}\label{terminalWPSQ}
    \hspace*{\algorithmicindent} \textbf{Input} \textit{List $L$} 
    \begin{algorithmic}[1]
    \Procedure{terminalWPSQ}{}
    \State $L_{aux} \gets \textit{empty list}$
    \State $L_p \gets \textit{powerset of  L}$
    \State $L_{sing} \gets \textit{singList(L)}$
    \State $k \gets 1$
                \While{$k \leq  \textit{length of}\,\, L_{sing}$}
				\If {$terminalSingQ(L,L_{sing}[k])$} 
				\State $L_{aux} \gets \textit{Append 0 to}\,\, L_{aux}$ \Else
				\State $L_{aux} \gets \textit{Append 1 to}\,\, L_{aux}$ 
				\EndIf
				\State $k \gets k+1$
				\EndWhile
    \State \If {$\sum_{x \in L_{aux}} x =0$} 
				\State $true$ \Else
				\State $false$; 
				\EndIf
    \EndProcedure
    \end{algorithmic}
    \end{algorithm}
\begin{Ex} \label{ex:nontermflip}
From Theorem \ref{thm:P3} only four triples - up to permutation - $(1,a,b)$ are the weights of a weighted blowup of $\mathbb{P}^3$ at a point which initiate a Sarkisov link. Indeed take the $(1,3,4)$-weighted blowup of a point. Then,
\begin{equation*}
\begin{array}{cccc|ccccc}
             &       & u     & x & y & z & t  & \\
\actL{T}   &  \lBr &  0 & 1 & 1 & 1 & 1  &  \actR{}\\
             &       & -1 & 0 & 1 & 3 & 4 &
\end{array}
\end{equation*}
is a terminal $\mathbb{Q}$-factorial variety isomorphic in codimension 1 to 

\begin{equation*}
\begin{array}{ccccc|cccc}
             &       & u     & x & y & z & t  & \\
\actL{T'}   &  \lBr &  0 & 1 & 1 & 1 & 1  &  \actR{}\\
             &       & -1 & 0 & 1 & 2 & 7 &
\end{array}
\end{equation*}
The isomorphism in codimension 1 is $\tau \colon T \dashrightarrow T'$ which contracts $\Gamma = \mathbb{P}^1$ and introduces $\Gamma' = \mathbb{P}(2,3)$. On $T'$, $\Gamma'$ has a $\frac{1}{2}(1,1,1)$ terminal singularity and a $\frac{1}{3}(1,1,1)$ canonical singularity.  The fact that there is a non-terminal singularity in $T'$ is captured by
\[
\text{terminalWPSQ}([-1,-1,2,3])=false.\\
\]
\end{Ex}

\bibliography{main}{}
\bibliographystyle{abbrv}

\end{document}